\documentclass[a4paper,11pt]{article}
 \usepackage[english]{babel}
\usepackage[a4paper]{geometry}
\usepackage{amsmath}
\usepackage{amsthm}
\usepackage{amssymb}
\usepackage{bbm}
\usepackage{color}
\usepackage[dvipsnames]{xcolor}

\usepackage{tikz}
\usetikzlibrary{positioning}
\usetikzlibrary{arrows}
\usepackage{caption}
\usepackage{subcaption}
 
\usepackage{mathrsfs}
\usepackage{enumerate}
 
\usepackage{hyperref} 

\def\bR{\mathbb{R}}
\def\bE{\mathbb{E}}
\def\bP{\mathbb{P}}

\def\cN{\mathcal{N}}

\def\wh{\widehat}

\def\be{\begin{equation}}
\def\ee{\end{equation}}

\def\rhs{r.h.s.\ }

\def\st{s.t.\ }

\newtheorem{theorem}{Theorem}   

\begin{document}
\title{On the Replica Symmetry of a Variant of the Sherrington-Kirkpatrick Spin Glass}
\author{Christian Brennecke\thanks{Institute for Applied Mathematics, University of Bonn, Endenicher Allee 60, 53115 Bonn, Germany} 
\and Adrien Schertzer\footnotemark[1]}
\maketitle
\begin{abstract}
We consider $N$ i.i.d.\ Ising spins with mean $m\in (-1,1)$ whose interactions are described by a Sherrington-Kirkpatrick Hamiltonian with a quartic correction. This model was recently introduced by Bolthausen in \cite{Bolt2} as a toy model to understand whether a second moment argument can be used to derive the replica symmetric formula in the full high temperature regime if $m\neq 0$. In \cite{Bolt2}, Bolthausen suggested that a natural analogue of the de Almeida-Thouless condition for the toy model is 
		\be\label{eq:conj} \beta^2(1-m^2)^2\leq 1. \ee
Here, $\beta \geq 0$ corresponds to the inverse temperature. While the second moment method implies replica symmetry for $\beta $ sufficiently small, Bolthausen showed that the method fails to prove replica symmetry in the full region described by \eqref{eq:conj}. A natural question that was left open in \cite{Bolt2} is whether \eqref{eq:conj} correctly characterizes the high temperature phase of the toy model. In this note, we show that this is indeed not the case. We prove that if $|m| \geq m_*$, for some $m_* \in (0,1)$, the limiting free energy of the toy model is negative for suitable $\beta $ that satisfy \eqref{eq:conj}.
\end{abstract}

\section{Introduction and Main Result} 
For $ \beta \geq 0$ and $m\in (-1,1)$, consider the mean field spin glass with partition function 
		\be \label{eq:model} Z_{N,\beta,m} := \sum_{\sigma\in \{-1,1\}^N}  p_m(\sigma) \exp \bigg(  \sum_{1\leq i , j\leq N} \frac{\beta}{\sqrt{2}} g_{ij} \wh \sigma_{i} \wh \sigma_j - \frac{\beta^2}{4N}  \sum_{1\leq i,j\leq N} \wh \sigma_{i}^2\wh \sigma_{j}^2\bigg). \ee
Here, $ \sigma = (\sigma_1, \ldots, \sigma_N)\in \{-1,1\}^N, \wh\sigma_i:= \sigma_i - m$, $p_m(\sigma) :=  \prod_{i=1}^N \frac12 (1+m\sigma_i)  $ and the couplings $ (g_{ij})_{1\leq i,j\leq N} $ are i.i.d.\ Gaussian random variables with distribution $\cN(0,N^{-1})$.   

The model in \eqref{eq:model} was recently introduced by Bolthausen in \cite{Bolt2}. The main result of \cite{Bolt2} is a new proof of the replica symmetric (RS) formula for the classical Sherrington-Kirkpatrick (SK) model \cite{SK} with external field strength $h\in\bR$ and for sufficiently small $\beta\geq 0$. The proof is based on a second moment argument, conditionally on an approximate solution of the TAP equations \cite{TAP} that was constructed in \cite{Bolt1}. The RS formula is expected to be true under the de Almeida-Thouless \cite{AT} condition 
		\be \label{eq:ATline}   \beta^2 E (1- m_{\beta,h}^2)^2 \leq 1, \; \text{ for }\; m_{\beta,h}=\text{tanh}(h+\beta\sqrt{q}Z), \;q= E \, \text{tanh}^2(h+\beta\sqrt{q}Z), \ee
where $Z$ is a standard Gaussian random variable. To date, this was verified in a very large subregion of the phase diagram in \cite{Ton, Tal1, Tal2, JagTob}, based on the analysis of the Parisi formula \cite{Par1, Par2} proved in \cite{Gue, Tal}. A complete proof that the RS phase is fully characterized by \eqref{eq:ATline} still remains open, however. Note for a centered Gaussian external field, the RS phase indeed turns out to be characterized by an analogue of \eqref{eq:ATline}, see \cite{Ch}.
 
It is a natural question whether a second moment argument as in \cite{Bolt2} can be improved to verify the replica symmetry in the expected region \eqref{eq:ATline}. Although the temperature range from \cite{Bolt2} can be improved through basic restriction arguments \cite{BY}, Bolthausen already suggested in \cite{Bolt2} that understanding the full RS phase likely requires additional techniques. His argument is based on the analysis of the toy model defined in \eqref{eq:model}. For this model, the second moment method can be evaluated explicitly and it is straightforward to show that for $\beta \geq 0$ small enough, one has that 
		\be \lim_{N\to\infty} \frac1N   \log  \bE Z_{N,\beta,m}  = \lim_{N\to\infty} \frac1N  \log \bE Z_{N,\beta,m}^2 =0.  \ee
By Gaussian concentration, see e.g.\ \cite[Th. 1.3.4.]{Tal1}, this implies for small $\beta $ that 
		\be \label{eq:1} \lim_{N\to\infty} \bE  \frac1N   \log  Z_{N,\beta,m} =0. \ee
By analogy to \eqref{eq:ATline}, Bolthausen then suggested that the high temperature region 
		\be \label{eq:setRSm}\text{HT}_m:= \Big\{\beta \geq 0: \lim_{N\to\infty} \bE  \frac1N   \log   Z_{N,\beta,m} = 0 \Big\}   \ee
the set in which the quenched equals the annealed free energy, should correspond to  
 		\be \label{eq:setATm} \text{AT}_m:= \big \{\beta \geq 0:  \beta^2(1-m^2)^2\leq 1\big \}  \ee
and it was left open whether \eqref{eq:1} indeed holds true for all $\beta \in \text{AT}_m$ (cf.\ \cite[Eq. (6.1)]{Bolt2}). 

Note that one might expect that $\text{AT}_m \subset \text{HT}_m $ based on different arguments. For example, an application of the TAP heuristics suggests that $\textbf{m} = (m, m, \ldots, m)\in (-1,1)^N$ solves an analogue of the TAP equations associated with \eqref{eq:model}. Alternatively, it is not difficult to see that, for $\beta$ sufficiently small, the free energy fluctuations converge to a normal variable with a variance that contains the term $\log \big(1- \beta^2(1-m^2)^2\big) $. This indicates a phase transition when $\beta^2(1-m^2)^2\to 1$, similarly as in the SK model without external field analyzed in detail in \cite{ALR} (in which case the corresponding divergence at $\beta \to 1$ detects the RS-RSB transition temperature correctly). 	

It turns out\footnote{We thank W.-K.\ Chen for this remark and for pointing out \cite{CHL} to us. We thank W.-K.\ Chen and S. Tang for sharing some related numerical simulations on \eqref{eq:setRSm} and \eqref{eq:setATm} with us.} that a precise characterization of $\text{HT}_m$ may be obtained by adapting the methods from \cite{CHL} to the present setting. Related numerical simulations suggest that for large $|m|$, one can find $\beta \in  \text{AT}_m$ with $\beta \not \in \text{HT}_m$. Perhaps surprisingly, this fact admits an elementary proof which does neither require precise knowledge of $ \text{HT}_m$ nor of a Parisi formula for the limiting free energy of the model. 	
\begin{theorem}\label{thm:main}  
Let $I(m):= - \frac{1+m}{2} \log \frac{1+m}{2} -  \frac{1-m}{2} \log \frac{1+m}{2}$ and define $ m_* >0$ to be such that $I(m_*) = \frac14 $. Then, for every $ |m| > m_*$, there exists $\beta \in \emph{AT}_m$ with $\beta \not \in \emph{HT}_m$. 
\end{theorem}


\section{Proof of Theorem \ref{thm:main}} \label{sec:proof}
Our proof is elementary and adapts an argument from the analysis of the random energy model, see \cite[p.\ 167]{Bov}, to the present setting. By symmetry, we assume in the rest of this section without loss of generality that $m>0$.  

\begin{proof}[Proof of Theorem \ref{thm:main}.] 
In the following, denote by $E_m$ the expectation with regards to the base measure $p_m$ and, for $\epsilon>0$, define the set
		\[ S_{N, m, \epsilon}:=\Big \{\sigma \in \{-1,1\}^N:   \frac 1N\sum_{i=1}^N \sigma_i \in [m-\epsilon,m+\epsilon]  \Big\}. \]
Moreover, denote by $ H_{N,\beta,m}$ the Hamiltonian
		\[H_{N,\beta,m}(\sigma) :=  \sum_{1\leq i , j\leq N} \frac{\beta}{\sqrt{2}} g_{ij} \wh \sigma_{i} \wh \sigma_j - \frac{\beta^2}{4N}  \sum_{1\leq i,j\leq N} \wh \sigma_{i}^2\wh \sigma_{j}^2 \]
so that $ Z_{N,\beta,m} = E_m \exp( H_{N,\beta,m})$. 

We first claim that for every $\beta \in  \text{HT}_m$ and every $\epsilon>0$, we have that
		\be\label{eq:claim1} \lim_{N\to\infty} \Big| \frac{1}{N} \log(Z_{N,\beta,m}) -  \frac{1}{N} \log E_{m}\left( \textbf{1}_{S_{N, m, \epsilon}}  \exp (H_{N,\beta,m})\right)\Big| =0 \ee
in the sense of probability. Here, $ \textbf{1}_{S}:\{-1,1\}^N\to \{0,1\} $ denotes the characteristic function of the subset $S \subset\{-1,1\}^N$. To prove \eqref{eq:claim1}, set 
		\[ Z_{N,\beta,m}^{ \leq \epsilon}:=E_{m} ( \textbf{1}_{S_{N, m, \epsilon}} \exp  H_{N,\beta,m} ), \; Z_{N,\beta,m}^{>\epsilon }:=E_{m}( \textbf{1}_{S_{N, m, \epsilon}^{\complement}} \exp H_{N,\beta,m} ) \]
so that $ Z_{N,\beta,m} = Z_{N,\beta,m}^{ \leq \epsilon} + Z_{N,\beta,m}^{ >\epsilon} $. Then, for every $\delta>0$, we have that
		\be\label{eq:m1}
		\begin{split}
		& \bP \Big(\Big| \frac{1}{N} \log(Z_{N,\beta,m})-\frac{1}{N}\log\big(Z_{N,\beta,m}^{ \leq \epsilon}\big)\Big|\geq \delta\Big)\\
		& = \bP \Big(   Z_{N,\beta,m}^{>\epsilon }/Z_{N,\beta,m}^{\leq \epsilon }  \geq  e^{\delta N} -1, \; Z_{N,\beta,m}^{ \leq \epsilon} \leq e^{-\delta N} \Big) \\
		&\hspace{0.4cm}+ \bP \Big(   Z_{N,\beta,m}^{>\epsilon }/Z_{N,\beta,m}^{\leq \epsilon }  \geq  e^{\delta N} -1, \; Z_{N,\beta,m}^{ \leq \epsilon} > e^{-\delta N} \Big)\\
		&\leq \bP \Big(    Z_{N,\beta,m}^{ \leq \epsilon} \leq e^{-\delta N} \Big) + \bP \Big(   Z_{N,\beta,m}^{>\epsilon }   \geq  1- e^{-\delta N}  \Big)\\
		&\leq \bP \Big(    Z_{N,\beta,m}^{ \leq \epsilon} \leq e^{-\delta N} \Big) +   \big(1- e^{-\delta N}\big)^{-1} p_m \big(  S_{N, m, \epsilon}^{\complement}\big), 
		\end{split}
		\ee
where in the last step, we applied Markov's inequality. Similarly, we have that 
 		\[\begin{split}
		&\bP \Big(    Z_{N,\beta,m}^{ \leq \epsilon} \leq e^{-\delta N} \Big) \\
		&\leq   \bP \Big( Z_{N,\beta,m}^{ \leq \epsilon} \leq e^{-\delta N}, \;  Z_{N,\beta,m}^{ >\epsilon} \leq e^{-\delta N} \Big) + \bP \Big(    Z_{N,\beta,m}^{ >\epsilon} >e^{-\delta N} \Big) \\
		& \leq  \bP \Big( \frac1N \log Z_{N,\beta,m}  \leq  - \frac\delta2  \Big) +  e^{\delta N} p_m \big(  S_{N, m, \epsilon}^{\complement}\big)
		\end{split}\]
for $N$ large enough so that 
		\[ \begin{split}
		&\bP \Big(\Big| \frac{1}{N} \log(Z_{N,\beta,m})-\frac{1}{N}\log\big(Z_{N,\beta,m}^{ \leq \epsilon}\big)\Big|\geq \delta\Big)\\
		& \leq  \bP \Big( \Big | \frac1N \log Z_{N,\beta,m} \Big|  \geq \frac\delta2  \Big) + \frac{e^{\delta N }}{1- e^{-\delta N}} \, p_m \big(  S_{N, m, \epsilon}^{\complement}\big).
		\end{split}\]	
Markov's inequality and Gaussian concentration (see \cite[Th. 1.3.4.]{Tal1}) imply that
		\[\begin{split}
		 \bP \Big( \Big | \frac1N \log Z_{N,\beta,m} \Big|  \geq \frac\delta2  \Big) &\leq   \bP \Big( \Big | \frac1N \log Z_{N,\beta,m} - \bE \frac1N \log Z_{N,\beta,m} \Big|  \geq \frac\delta4  \Big) \!+ \frac{4}{\delta}  \Big|   \bE \frac1N \log Z_{N,\beta,m} \Big| \\
		&\leq   C e^{-cN\delta^{2}/\beta^{2}}  + \frac{4}{\delta} \Big|   \bE \frac1N \log Z_{N,\beta,m} \Big| 
		\end{split}\]
for suitable constants $C, c>0$ (independent of all parameters). 
On the other hand, setting $ h := \tanh^{-1}(m)$, a standard exponential moment bound shows that  
		\[\begin{split}
		p_m \big( S_{N, m, \epsilon}^{\complement}\big) & \leq \inf_{\lambda  \geq 0}   \exp \big( -N\epsilon \lambda - N m \lambda + N (\log\cosh(h+\lambda) -\log\cosh(h)  \big) \\
		&\hspace{0.5cm} + \inf_{\lambda  \geq 0}  \exp \big( -N\epsilon \lambda + N m \lambda + N (\log\cosh(h-\lambda) -\log\cosh(h)   \big)\\
		&\leq 2  \inf_{\lambda  \geq 0}   \exp \big( -N\epsilon \lambda + N \lambda^2/2\big)  = 2 \exp ( - N\epsilon^2/2  ).
		\end{split}\]
Now, for every $  \delta'>0$, we can find $ 0<\delta< \min (\delta', \epsilon^2/4)$ so that we arrive at
		\[ \begin{split}
		&\limsup_{N\to\infty} \bP \Big(\Big| \frac{1}{N} \log(Z_{N,\beta,m})-\frac{1}{N}\log\big(Z_{N,\beta,m}^{ \leq \epsilon}\big)\Big|\geq \delta'\Big)\\
		& \leq \limsup_{N\to\infty} \bP \Big(\Big| \frac{1}{N} \log(Z_{N,\beta,m})-\frac{1}{N}\log\big(Z_{N,\beta,m}^{ \leq \epsilon}\big)\Big|\geq \delta\Big)\\
		& \leq \limsup_{N\to\infty}\bigg(  C e^{-cN\delta^{2}/\beta^{2}} +  C e^{  -N\epsilon^2/4 }   +  \frac{4}{\delta}  \Big|   \bE \frac1N \log Z_{N,\beta,m} \Big|\bigg) \!= \frac{4}{\delta} \limsup_{N\to\infty}   \Big|   \bE \frac1N \log Z_{N,\beta,m} \Big|. 
		\end{split}\]
By definition of $\text{HT}_m$ in \eqref{eq:setRSm}, the \rhs vanishes if $ \beta \in \text{HT}_m$ and this concludes \eqref{eq:claim1}. 
		
Notice that \eqref{eq:claim1} implies that for every $\beta\in \text{HT}_m$ and every $\epsilon>0$, we have that
		\be \label{eq:RScond} \begin{split}
		\lim_{N\to\infty}  \frac1N \log Z_{N,\beta, m}^{\leq \epsilon} =  \lim_{N\to\infty}  \bE \frac1N \log Z_{N,\beta, m}^{\leq \epsilon} =0, 
		\end{split}\ee
where the first equality holds true almost surely. Indeed, this follows from \eqref{eq:claim1} combined with Gaussian concentration applied to both $ Z_{N,\beta, m}$ and $Z_{N,\beta, m}^{\leq \epsilon} $. 

We now argue that for every $m>m_*$, we can find $ \beta \in \text{AT}_m$ such that for all $\epsilon>0$ sufficiently small, it holds true that
		\be \label{eq:claim2} \begin{split}
		  \limsup_{N\to\infty}  \bE \frac{1}{N} \log Z_{N,\beta,m}^{ \leq \epsilon}<0.
		\end{split}\ee
By \eqref{eq:RScond}, this implies $ \beta \not\in \text{HT}_m$ and it concludes that $\text{AT}_m\not \subset \text{HT}_m$ whenever $m>m_*$. 

In the remainder, we abbreviate
 		$$F_{N,\beta,m}^{ \leq \epsilon}:=\bE \frac{1}{N} \log Z_{N,\beta,m}^{ \leq \epsilon}.$$ 
Then, we have that 
		\[\begin{split}
		\frac{d}{d\beta} F_{N,\beta,m}^{ \leq \epsilon} &=\bE \,  E_{m}\bigg( \frac1N \sum_{1\leq i , j\leq N} \frac{g_{ij}}{\sqrt{2}} {\widehat \sigma_{i}}{\widehat \sigma_j}-\frac{\beta }{2}\Big(\frac{1}{N}\sum_{i=1}^{N}{\widehat \sigma_i}^2\Big)^2  \bigg)   \textbf{1}_{S_{N, m, \epsilon}}  \frac{e^{ H_{N,\beta,m} } } {Z_{N,\beta,m}^{ \leq \epsilon}} \\
		& = \bE \,  E_{m} \, \frac1N \sum_{1\leq i , j\leq N} \frac{g_{ij}}{\sqrt{2}} {\widehat \sigma_{i}}{\widehat \sigma_j}  \, \textbf{1}_{S_{N, m, \epsilon}}  \frac{e^{ H_{N,\beta,m} } } {Z_{N,\beta,m}^{ \leq \epsilon}} -\frac{\beta}{2}(1-m^2)^2+C \beta \epsilon\\
		&\leq \frac{ \bE M_{N,m,\epsilon}}{\sqrt N}-\frac{\beta}{2}(1-m^2)^2+ C \beta \epsilon.
		\end{split}\]
Here, $C>0$ is some universal constant and we defined
		\[\begin{split} 
		M_{N,m,\epsilon}:=  \max \big\{ X_\sigma :  \sigma\in S_{N, m, \epsilon} \big\} \;\;\text{ for } \;\; X_\sigma:= \frac{1}{\sqrt {2N}}  \sum_{1\leq i , j\leq N} g_{ij}  {\widehat \sigma_{i}}{\widehat \sigma_j}.
		\end{split} \]
To control the maximum $ M_{N,m,\epsilon}$, we apply Jensen's inequality so that for $t>0$, we get
		\[\begin{split}
		 \frac{ \bE M_{N,m,\epsilon}}{\sqrt N}& \leq \frac1{Nt} \bE \log \exp\big (\sqrt{N} tM_{N,m,\epsilon}\big)\\
		&\leq  \frac1{Nt} \log \bE\exp\big (\sqrt{N} t M_{N,m,\epsilon}\big)\\
		&\leq  \frac1{Nt} \log   \bE \sum_{\sigma \in S_{N, m, \epsilon}} \exp\big (\sqrt{N} tX_\sigma \big)  =  \frac1{Nt}  \log  \sum_{\sigma \in S_{N, m, \epsilon}}  \exp\Big (\frac{N t^2 }{2}  \bE X_\sigma^2 \Big).
		\end{split}	\]
Now, recall that for every $\sigma \in S_{N, m, \epsilon}$, we have that
		\[ \bE X_\sigma^2  =  \frac12( 1+m^2- 2m s_\sigma )^2 \;\text{ for some } \; s_\sigma \in \Delta_N \cap [m-\epsilon, m+\epsilon], \]
where $\Delta_N:=\big\{ -1  +2k/N :  k = 0, 1,\ldots, N  \big\}$. Moreover, by Stirling's approximation, we have for every $s\in \Delta_N \cap [m-\epsilon, m+\epsilon]$ (and $\epsilon>0$ small enough) that
		\[  \big| \big\{ \sigma \in \{-1,1\}^N: s_\sigma  = s    \big\} \big| = \binom{N}{  \frac{1+s}{2} N } \leq C \exp\big(N   I(s)  \big)     \]
for $N$ large enough, a universal constant $C>0$ and where we recall that $I: [-1,1]\to [0,1]$ denotes the binary entropy function, defined by
		\[ \label{eq:bient} I(s)= - \frac{1+s}{2} \log \frac{1+s}{2}  -  \frac{1-s}{2} \log \frac{1-s}{2} . \]
Recall that $I$ is a strictly concave function, that it attains its unique maximum at $s=0$, vanishes at $ s= \pm 1$ and that it is strictly decreasing in $ [0,1]$. Using that \linebreak $ |\Delta_N \cap [m-\epsilon, m+\epsilon]|\leq  N $, we obtain for every $t>0$ the upper bound
		\[\begin{split}
		\frac{ \bE M_{N,m,\epsilon}}{\sqrt N} \leq \frac1{t} \frac{ \log (C N ) }N + \frac 1 t  \max_{s  \in  [m-\epsilon, m+\epsilon] }  \Big (\frac{t^2}{4}  ( 1+m^2- 2m s  )^2 +  I(s) \Big). 
		\end{split}\]
Choosing $ t:=  2 (1-m^2)^{-1}  \sqrt {I(m)}  $ \st in particular $t>0$ for $m\in (0,1)$, this implies
		\[\begin{split}
		\frac{ \bE M_{N,m,\epsilon}}{\sqrt N} & \leq \frac{ (1-m^2)  }{2\sqrt {I(m)}}  \frac{  \log (C N )}{N  }  \\
		&\hspace{0.5cm} +  \frac{ (1-m^2)  }{ 2\sqrt {I(m)}}  \max_{s  \in  [m-\epsilon, m+\epsilon] } \Big ( I(m) (1-m^2)^{-2} ( 1+m^2- 2m s  )^2 + I(s) \Big) \\
		&\leq  \frac{ (1-m^2)  }{ 2\sqrt {I(m)}}  \frac{  \log (C N )}{N  } +  (1-m^2)  \sqrt {I(m)} + C\epsilon 
		\end{split}\]
for some $C>0$ (that is independent of $\beta$) and for all $\epsilon>0$ small enough, by continuity. 

Collecting the previous estimates, we conclude that for every $m\in (0,1)$, every $\beta \geq 0$ and for every sufficiently small $\epsilon>0$ it holds true that
		\[\frac{d}{d\beta} F_{N,\beta,m}^{ \leq \epsilon} \leq     (1-m^2)   \sqrt {I(m)} -\frac{\beta}{2}(1-m^2)^2 + C(1+\beta)\epsilon + O(\log N/N).    \]
Integrating the last inequality in $\beta \geq 0$ and combining this with $ F_{N,\beta,m}^{ \leq \epsilon} \leq 0$ for all $\beta\geq 0$, which follows from Jensen's inequality, we obtain the upper bound
		\[
		\limsup_{\epsilon\to0}\limsup_{N\to\infty} F_{N,\beta,m}^{ \leq \epsilon}\leq \inf_{\beta_0\geq 0}
		\begin{cases} 
		0 &: \beta <\beta_0, \\
		(\beta-\beta_0) (1-m^2) \sqrt {I(m)} -\frac{\beta^2- \beta_0^2}{4}(1-m^2)^2 &: \beta \geq \beta_0. 
		\end{cases}
		\]
 Choosing $ \beta_c= 2(1-m^2)^{-1} \sqrt {I(m)} >0$ to bound the \rhs from above, we arrive at
		\be\label{eq:lwrbnd}
		\limsup_{\epsilon\to0}\limsup_{N\to\infty} F_{N,\beta,m}^{ \leq \epsilon}\leq  
		\begin{cases} 
		0 &: \beta <\beta_c, \\
		 -\frac14 (\beta-\beta_c)^2 (1-m^2)^2  &: \beta \geq \beta_c. 
		\end{cases}
		\ee		
To conclude the proof, we finally notice that \eqref{eq:lwrbnd} implies that 
		$$\limsup_{\epsilon\to0}\limsup_{N\to\infty} F_{N,\beta,m}^{ \leq \epsilon}<0 $$ 
whenever 
 		\be \beta > \beta_c \;\;\Longleftrightarrow  \;\; \beta^2 (1-m^2)^2 > 4 I(m).  \ee 
 If $ m > m_*  $, for $m_*>0$ defined by $I(m_*)=\frac14$, the monotonicity of $I$ implies that $4 I(m)< 1 $ so that, in particular, there exists some $\beta \in \text{AT}_m$ that satisfies 
 		\[  4 I(m) <  \beta^2 (1-m^2)^2 \leq 1  \;\;\Longleftrightarrow  \;\;   \beta  \in\big(  (1-m^2)^{-1} \sqrt{4I(m)} , (1-m^2)^{-1}\big]. \]
Choosing $\epsilon>0$ small enough, this implies \eqref{eq:claim2} and concludes that $\text{AT}_m \not \subset \text{HT}_m$. 
\end{proof}

\vspace{0.2cm}
\noindent\textbf{Acknowledgements.}  C. B. thanks G. Genovese for helpful discussions. C. B. and A. S. acknowledge support by the Deutsche Forschungsgemeinschaft (DFG, German Research Foundation) under Germany’s Excellence Strategy – GZ 2047/1, Projekt-ID 390685813. A.S. is also funded by the Deutsche Forschungsgemeinschaft – project number 432176920. 



\begin{thebibliography}{55}

\bibitem{ALR} M. Aizenman, J. L. Lebowitz, D. Ruelle. Some rigorous results on the Sherrington-Kirkpatrick spin glass model. \emph{Comm. Math. Phys.} \textbf{11}, pp. 3-20 (1987).




\bibitem{Bolt1} E. Bolthausen. An Iterative Construction of Solutions of the TAP Equations for the Sherrington–Kirkpatrick Model. \emph{Comm. Math. Phys.} \textbf{325}, pp. 333-366 (2014).

\bibitem{Bolt2} E. Bolthausen. A Morita Type Proof of the replica-symmetric Formula for SK. In: \emph{Statistical Mechanics of Classical and Disordered Systems}, Springer Proceedings in Mathematics \& Statistics, pp. 63-93 (2018).

\bibitem{Bov} A. Bovier. Statistical Mechanics of Disordered Systems: A Mathematical Perspective. \emph{Cambridge Series in Statistical and Probabilistic Mathematics} (2012).



\bibitem{BY} C. Brennecke, H.-T. Yau. The Replica Symmetric Formula for the SK Model Revisited. \emph{J. Math. Phys.} \textbf{63}, 2022.


\bibitem{Ch} W.-K. Chen. On the Almeida-Thouless transition line in the SK model with centered Gaussian external field. \emph{Electron. Commun. Probab.} \textbf{26}, pp. 1-9 (2021).

\bibitem{CHL} W.-K. Chen, M. Handschy, G. Lerman. Phase transition in random tensors with multiple spikes. \emph{Ann. App. Probab.} \textbf{31}, No. 4, 1868-1913 (2021).

\bibitem{AT} J. R. L. de Almeida, D. J. Thouless. Stability of the Sherrington-Kirkpatrick solution of a spin glass model. \emph{J. Phys. A: Math. Gen.} \textbf{11}, pp. 983-990 (1978).


\bibitem{Gue} F. Guerra. Broken Replica Symmetry Bounds in the Mean Field Spin Glass Model. \emph{Comm. Math. Phys.} \textbf{233}, pp. 1-12 (2003).

\bibitem{JagTob} A. Jagannath, I. Tobasco. Some properties of the phase diagram for mixed p-spin glasses. \emph{Probab. Theory Relat. Fields} \textbf{167}, pp. 615-672 (2017).





\bibitem{Par1} G. Parisi. Infinite number of order parameters for spin-glasses. \emph{Phys. Rev. Lett.} \textbf{43}, pp. 1754-1756 (1979).

\bibitem{Par2} G. Parisi. A sequence of approximate solutions to the S-K model for spin glasses. \emph{J. Phys. A} \textbf{13}, pp. L-115 (1980).


\bibitem{SK} D. Sherrington, S. Kirkpatrick. Solvable model of a spin glass. \emph{Phys. Rev. Lett.} \textbf{35}, pp. 1792-1796 (1975).

\bibitem{Slepian} D. Slepian. The one-sided barrier problem for Gaussian noise. \emph{Bell System Tech. J.}, \textbf{41}, 463–501 (1962).

\bibitem{Tal} M. Talagrand. The Parisi formula. \emph{Ann. of Math.} \textbf{163}, no. 1, pp. 221-263 (2006).

\bibitem{Tal1} M. Talagrand. Mean Field Models for Spin Glasses. Volume I: Basic Examples. A Series of Modern Surveys in Mathematics {\bf 54}, \emph{Springer Berlin, Heidelberg} (2010).

\bibitem{Tal2} M. Talagrand. Mean Field Models for Spin Glasses. Volume II: Advanced Replica-Symmetry and Low Temperature. A Series of Modern Surveys in Mathematics \textbf{54}, \emph{Springer Verlag Berlin-Heidelberg} (2011).

\bibitem{TAP} D. J. Thouless, P. W. Anderson, R. G. Palmer. Solution of 'solvable model in spin glasses'. \emph{Philos. Magazin} \textbf{35}, pp. 593-601 (1977).
  
\bibitem{Ton} F. L. Toninelli. About the Almeida-Thouless transition line in the Sherrington- Kirkpatrick mean-field spin glass model. \emph{Europhys. Lett.} \textbf{60} (5), p. 764 (2002). 
  
\end{thebibliography}
\end{document}